\documentclass[12pt, twoside]{article}
\usepackage{secdot}

\usepackage{amsmath, amsthm, amscd, amsfonts, amssymb, graphicx, color}
\usepackage[bookmarksnumbered, colorlinks, plainpages]{hyperref}

\begin{document}
\centerline{}

\centerline{\Large{\bf Cotton Solitons on Almost Kenmotsu 3-$h$-Manifolds}}
\vspace{0.2cm}

\centerline{}

\newcommand{\mvec}[1]{\mbox{\bfseries\itshape #1}}

\centerline{\large{\textbf{Dibakar Dey and Pradip Majhi}}}
\vspace{0.1cm}
\centerline{Department of Pure Mathematics} 
\centerline{ University of Calcutta,
35, Ballygunge Circular Road}
\centerline{Kolkata-700019, West Bengal, India }
\centerline{E-mail: deydibakar3@gmail.com and Pradip Majhi}



\newtheorem{theorem}{\quad Theorem}[section]

\newtheorem{definition}[theorem]{\quad Definition}

\newtheorem{proposition}[theorem]{\quad Proposition}

\newtheorem{question}[theorem]{\quad Question}

\newtheorem{remark}[theorem]{\quad Remark}

\newtheorem{corollary}[theorem]{\quad Corollary}

\newtheorem{note}[theorem]{\quad Note}

\newtheorem{lemma}[theorem]{\quad Lemma}
\newtheorem{example}[theorem]{\quad Example}
\newtheorem{notation}[theorem]{\quad Notation}
\numberwithin{equation}{section}
\newcommand{\be}{\begin{equation}}
\newcommand{\ee}{\end{equation}}
\newcommand{\bea}{\begin{eqnarray}}
\newcommand{\eea}{\end{eqnarray}}

\vspace{0.3 cm}
\textbf{Abstract:} In this paper, we consider the notion of Cotton soliton within the framework of almost Kenmotsu 3-$h$-manifolds. First we consider that the potential vector field is pointwise collinear with the Reeb vector field and prove a non-existence of such Cotton soliton. Next we assume that the potential vector field is orthogonal to the Reeb vector field. It is proved that such a Cotton soliton on a non-Kenmotsu almost Kenmotsu 3-$h$-manifold such that the Reeb vector field is an eigen vector of the Ricci operator is steady and the manifold is locally isometric to $\mathbb{H}^2(-4) \times \mathbb{R}$.\\



\section{Introduction}
An almost contact metric manifold is an odd dimensional differentiable manifold $M^{2n+1}$ together with a structure $(\varphi,\xi,\eta,g)$ satisfying (\cite{bl}, \cite{bll})
\be
\varphi^{2}X = - X + \eta(X)\xi,\;\; \eta(\xi)=1, \label{1.1}
\ee
\be
 g(\varphi X,\varphi Y) = g(X,Y) - \eta(X)\eta(Y) \label{1.2}
\ee
for all vector fields $X$, $Y$ on $M^{2n+1}$, where $g$ is the Riemannian metric, $\varphi$ is a $(1,1)$-tensor field, $\xi$ is a unit vector field called the Reeb vector field and $\eta$ is a 1-form defined by $\eta(X) = g(X,\xi)$. Here also $\phi\xi=0$ and $\eta\circ\phi=0$; both can be derived from \eqref{1.1} easily. The fundamental 2-form $\Phi$ on an almost contact metric manifold is defined by $\Phi(X,Y)=g(X,\varphi Y)$ for all vector fields $X$, $Y$ on $M^{2n+1}$. The condition for an almost contact metric manifold being normal is equivalent to vanishing of the $(1,2)$-type torsion tensor $N_{\varphi}$, defined by $N_{\varphi}=[\varphi,\varphi]+2d\eta\otimes\xi$, where $[\varphi,\varphi]$ is the Nijenhuis tensor of $\varphi$ \cite{bl}. Almost contact metric manifold such that $\eta$ is closed and $d\Phi=2\eta\wedge\Phi$  are called almost Kenmotsu manifolds (see \cite{dp},\cite{dp2}). Obviously, a normal almost Kenmotsu manifold is a Kenmotsu manifold. Also Kenmotsu manifolds can be characterized by $(\nabla_{X}\varphi)Y = g(\varphi X,Y)\xi-\eta(Y)\varphi X$, for any vector fields $X$, $Y$ on $M^{2n+1}$.\\

The $(0,3)$-Cotton tensor of a 3-dimensional Riemannian manifold $(M^3,g)$ is defined as
\bea
\nonumber C(X,Y,Z) &=& (\nabla_X S)(Y,Z) - (\nabla_Y S)(X,Z) \\ && - \frac{1}{4}[(X(r)g(Y,Z) - Y(r)g(X,Z)], \label{1.3}
\eea
where $S$ is the Ricci tensor and $r$ is the scalar curvature of $M^3$. The Cotton tensor is skew-symmetric in first two indices and totally trace free. It is well known that for $n \geq 4$, an $n$-dimensional Riemannian manifold is conformally flat if the Weyl tensor vanishes. For $n = 3$, the Weyl tensor always vanishes but the Cotton tensor does not vanish in general.\\

In 2008, Kicisel, Sario\u{g}lu and Tekin \cite{tekin} introduced the notion of Cotton flow as a analogy of the Ricci flow. The Cotton flow is based on the conformally invariant Cotton tensor and defined exclusively foe 3-simension as
\bea
\nonumber \frac{\partial g}{\partial t} = C,
\eea
where $C$ is the $(0,2)$-Cotton tensor of $g$. From the Cotton flow, they defined the notion of Cotton soliton as follows:
\begin{definition}
A Cotton soliton is a metric $g$ defined on 3-dimensional smooth manifold $M^3$ such that the following equation
\bea
(\mathcal{L}_V g)(X,Y) + C(X,Y) - \sigma g(X,Y) = 0, \label{1.4}
\eea
holds for a constant $\sigma$ and a vector field $V$, called the potential vector field, where $\mathcal{L}_V$ denotes the Lie derivative along $V$ and $C$ is the $(0,2)$-Cotton tensor defined by
\bea
C_{ij} = \frac{1}{2\sqrt{g}}C_{nmi}\epsilon^{nml}g_{lj} \label{1.5}
\eea
in a local frame of $M^3$, where $g = \operatorname{det}(g_{ij})$, $C_{ijk}$ is the $(0,3)$-Cotton tensor and $\epsilon$ is a tensor density.
\end{definition}
In an orthonormal frame, $\epsilon^{123} = 1$. Also exchange of any two indices will give rise to minus sign and it will be zero if there has two same indices. For example, $\epsilon^{231} = - \epsilon^{213}$ and $\epsilon^{112} = \epsilon^{122} = \epsilon^{223} = 0$. Cotton solitons are fixed points of the Cotton flow upto diffeomorphisms and rescaling. The Cotton soliton is said to be shrinking, steady or expanding according as $\sigma$ is positive, zero or negative respectively. As far as we know, the Cotton soliton was studied by Chen \cite{chen} on certain almost contact metric manifold, precisely on almost coK\"{a}hler 3-manifolds. Motivated by the study of Chen \cite{chen}, we consider the notion of Cotton soliton on an almost Kenmotsu 3-$h$-manifold and prove some related results.

\section{Almost Kenmotsu 3-$h$-Manifolds}

Let $(M^3,\varphi,\xi,\eta,g)$ be a 3-dimensional almost Kenmotsu manifold. We denote by $l = R(\cdot,\xi)\xi$, $h = \frac{1}{2}\mathcal{L}_\xi \varphi$ and $h' = h \circ \varphi$ on $M^3$, where $R$ is the Riemannian curvature tensor. The tensor fields $l$ and $h$ are symmetric operators and satisfy the following relations (\cite{dp}, \cite{dp2}):
 \bea
h\xi = 0,\;l\xi = 0,\;tr(h) = 0,\;tr(h\varphi) = 0,\;h\varphi+\varphi h = 0, \label{2.1}
 \eea
\bea
 \nabla_{X}\xi = X - \eta(X)\xi - \varphi hX(\Rightarrow \nabla_{\xi}\xi=0), \label{2.2}
\eea
\bea
\nabla_\xi h = - \varphi - 2h - \varphi h^2 - \varphi l. \label{2.3}
\eea
\begin{definition} $\cite{ww}$
A 3-dimensional almost Kenmotsu manifold is called an almost Kenmotsu 3-$h$-manifold if it satisfies $\nabla_\xi h = 0$.
\end{definition}

Let $\mathcal{U}_1$ be the maximal open subset of a 3-dimensional almost Kenmotsu manifold $M^3$ such that $h \neq 0$ and $\mathcal{U}_2$ be the maximal open subset on which $h = 0$. Then $\mathcal{U}_1 \cup \mathcal{U}_2$ is an open and dense subset of $M^3$. Then $\mathcal{U}_1$ is non-empty and there is a local orthonormal basis $\{e_1 = \xi, e_2 = e, e_3 = \varphi e\}$ on $\mathcal{U}_1$ such that $he = \lambda e$ and $h\varphi e = - \lambda \varphi e$ for some positive function $\lambda$. Since $\nabla_\xi h = 0$ for an almost Kenmotsu 3-$h$-manifold, then using Lemma 6 of \cite{cho} and \eqref{2.3}, a direct calculation gives $\xi(\lambda) = a = 0$. Therefore, Lemma 6 of \cite{cho} can be rewritten for an almost Kenmotsu 3-$h$-manifold as
\begin{lemma} \label{lem2.2}
On $\mathcal{U}_1$, the coefficients of the Riemannian connection $\nabla$ of an almost Kenmotsu 3-$h$-manifold with respect to a local orthonormal basis $\{\xi,e,\varphi e\}$ is given by
$$\nabla_\xi \xi = 0,\;\; \nabla_\xi e = 0,\;\; \nabla_\xi \phi e = 0,$$
$$\nabla_e \xi = e - \lambda \varphi e,\;\; \nabla_e e = - \xi - b\varphi e,\;\; \nabla_e \varphi e = \lambda \xi + be,$$
$$\nabla_{\varphi e} \xi = - \lambda e + \varphi e,\;\; \nabla_{\varphi e} e = \lambda \xi + c\varphi e,\;\; \nabla_{\varphi e} \varphi e = - \xi - ce,$$
where $b$ and $c$ are smooth functions.
\end{lemma}
From Lemma \ref{lem2.2}, the Lie brackets can be calculated as follows:
\bea
[e,\xi] = e - \lambda \varphi e,\;\; [e,\varphi e] = be - c\varphi e\;\; \mathrm{and}\;\; [\varphi e,\xi] = - \lambda e + \varphi e. \label{2.4}
\eea
From Lemma 3.2 of \cite{wang2}, the Ricci operator $Q$ of an almost Kenmotsu 3-$h$-manifold can be obtained as
\begin{lemma} \label{lem2.3}
On $\mathcal{U}_1$, the Ricci operator of an almost Kenmotsu 3-$h$-manifold with respect to a local orthonormal basis $\{\xi,e,\varphi e\}$ is given by
$$Q\xi = -2(\lambda^2 + 1)\xi - [\varphi e(\lambda) + 2\lambda b]e - [e(\lambda) + 2\lambda c]\varphi e,$$
$$Qe = - [\varphi e(\lambda) + 2\lambda b]\xi - fe + 2\lambda \varphi e,$$
$$Q\varphi e = - [e(\lambda) + 2\lambda c]\xi + 2\lambda e - f\varphi e,$$
where $f = e(c) + \varphi e(b) + b^2 + c^2 + 2$.
\end{lemma}
The scalar curvature $r$ of an almost Kenmotsu 3-$h$-manifold is given by
\bea
r = g(Qe_i,e_i) = - 2(\lambda^2 + 1) - 2f. \label{2.5}
\eea
Using Lemma \ref{lem2.3}, we obtain
\bea \label{2.6}
\begin{cases}
  S(\xi,\xi) = -2(\lambda^2 + 1),\;  S(\xi,e) = - [\varphi e(\lambda) + 2\lambda b], \\
  S(\xi,\varphi e) = - [e(\lambda) + 2\lambda c],\;  S(e,e) = -f, \\
   S(e,\varphi e) = 2\lambda, \; S(\varphi e,\varphi e) = - f. 
 \end{cases}
\eea
It is well known that an almost Kenmotsu 3-manifold is Kenmotsu if and only if $h = 0$. Thus a Kenmotsu metric always admits a almost Kenmotsu 3-$h$-metric structure. We now give an example of a non-Kenmotsu almost Kenmotsu 3-$h$-manifold.
\begin{example} $\cite{wa2}$
Let $M^3$ be a 3-dimensional non-unimodular Lie group with a left invariant local orthonormal frame $\{e_1,e_2,e_3\}$ satisfying
$$[e_1,e_2] = \alpha e_2 + \beta e_3,\; [e_2,e_3] = 0 \; \mathrm{and} \; [e_1,e_3] = \beta e_2 + (2 - \alpha)e_3$$
for $\alpha, \beta \in \mathbb{R}$. If either $\alpha \neq 1$ or $\beta \neq 0$, then $M^3$ admits a non-Kenmotsu almost Kenmotsu 3-$h$-metric struicture. 
\end{example}

\section{Cotton Soliton}
In this section, we consider the notion of Cotton soliton within the framework of almost Kenmotsu 3-$h$-manifolds. To study the notion of Cotton soliton, we need to compute the components of the $(0,2)$-Cotton tensor. In this regard, we prove the following Lemma:
\begin{lemma} \label{lem3.1}
The components of the $(0,2)$-Cotton tensor $C$ with respect to an orthonormal frame $\{\xi,e,\varphi e\}$ of a non-Kenmotsu almost Kenmotsu 3-$h$-manifold $M^3$ can be expressed as follows:
\bea
\nonumber  C_{11} = C(\xi,\xi) &=& b[\varphi e(\lambda) + 2\lambda b] - c[e(\lambda) + 2\lambda c] \\ && - e(e(\lambda) + 2\lambda c) + \varphi e(\varphi e(\lambda) + 2\lambda b), \label{3.1}
\eea
\bea
\nonumber  C_{12} = C(\xi,e) &=& 2[e(\lambda) -3\lambda \varphi e(\lambda) + 2\lambda c - 2\lambda^2 b] \\ && + \xi(e(\lambda) + 2\lambda c) - \frac{1}{4}\varphi e(r), \label{3.2}
\eea
\bea
\nonumber  C_{13} = C(\xi,\varphi e) &=& -2[\varphi e(\lambda) -3\lambda e(\lambda) + 2\lambda b - 2\lambda^2 c] \\ && -\xi(\varphi e(\lambda) + 2\lambda b) + \frac{1}{4}e(r), \label{3.3}
\eea
\bea
  C_{22} = C(e,e) = 2\lambda^3 - f\lambda +c[e(\lambda) + 2\lambda c]  - \varphi e(\varphi e(\lambda) + 2\lambda b), \label{3.4}
\eea
\be
C_{23} = C(e,\varphi e) = - \xi(f) - f + 2 + e(\varphi e(\lambda) + 2\lambda b) + b[e(\lambda) + 2\lambda c] - \frac{1}{4}\xi(r), \label{3.5}
\ee
\be
C_{33} = C(\varphi e,\varphi e) = - 2\lambda^3 + f\lambda - b[\varphi e(\lambda) + 2\lambda b] + e(e(\lambda) + 2\lambda c). \label{3.6}
\ee
\end{lemma}
\begin{proof}
The components of the metric tensor $g$ with respect to an orthonormal frame $\{\xi,e,\varphi e\}$ of a non-Kenmotsu almost Kenmotsu 3-$h$-manifold $M^3$ is given by
\bea
\nonumber (g_{ij}) = \left( \begin{array}{ccc} 1 & 0 & 0 \\ 0 & 1 & 0 \\ 0 & 0 & 1 \end{array} \right)
\eea
and hence $\operatorname{det}(g_{ij}) = 1$. Therefore, equation \eqref{1.5} reduces to
\bea
\nonumber C_{ij} = \frac{1}{2}C_{nmi}\epsilon^{nmj},\;\; i,j = 1,2,3,
\eea
where $C_{ijk} = C(e_i,e_j,e_k)$. Also, $C_{ijk} = - C_{jik}$ and $C_{iik} = 0$ for all $i,j,k = 1,2,3$. It can be easily obtained that (see \cite{chen})
\bea
\nonumber C_{11} = C_{231},\; C_{12} = C_{311},\; C_{13} = C_{121},\; C_{22} = C_{312},\; C_{23} = C_{122},\; C_{33} = C_{123}.
\eea
Making use of \eqref{1.3}, we get the followings:
\bea
 C_{11} = C_{231} = C(e,\varphi e,\xi) = (\nabla_e S)(\varphi e,\xi) - (\nabla_{\varphi e} S)(e,\xi), \label{3.7}
\eea
\bea
C_{12} = C(\varphi e,\xi,\xi) = (\nabla_{\varphi e} S)(\xi,\xi) - (\nabla_\xi S)(\varphi e,\xi) - \frac{1}{4}\varphi e(r), \label{3.8}
\eea
\bea
C_{13} = C(\xi,e,\xi) = (\nabla_\xi S)(e,\xi) - (\nabla_e S)(\xi,\xi) + \frac{1}{4}e(r), \label{3.9}
\eea
\bea
C_{22} = C(\varphi e,\xi,e) = (\nabla_{\varphi e} S)(\xi,e) - (\nabla_\xi S)(\varphi e,e), \label{3.10}
\eea
\bea
C_{23} = C(\xi,e,e) = (\nabla_\xi S)(e,e) - (\nabla_e S)(\xi,e) - \frac{1}{4}\xi(r), \label{3.11}
\eea
\bea
 C_{33} = C(\xi,e,\varphi e) = (\nabla_\xi S)(e,\varphi e) - (\nabla_e S)(\xi,\varphi e). \label{3.12}
\eea
Using \eqref{2.6}, Lemma \ref{lem2.2} and $\xi(\lambda) = 0$, we now obtain the followings:
\bea
\begin{cases}
(\nabla_e S)(\varphi e,\xi) = 2\lambda^3 - f\lambda + b[\varphi e(\lambda) + 2\lambda b] - e(e(\lambda) + 2\lambda c),\\
(\nabla_{\varphi e} S)(e,\xi) = 2\lambda^3 - f\lambda + c[e(\lambda) + 2\lambda c] -\varphi e(\varphi e(\lambda) + 2\lambda b). \label{3.13}
\end{cases} 
\eea
\bea
\begin{cases}
(\nabla_{\varphi e} S)(\xi,\xi) = 2[e(\lambda) -3\lambda \varphi e(\lambda) + 2\lambda c - 2\lambda^2 b],\\
(\nabla_\xi S)(\varphi e,\xi) = - \xi(e(\lambda) + 2\lambda c). \label{3.14}
\end{cases} 
\eea
\bea
\begin{cases}
(\nabla_\xi S)(e,\xi) = -\xi(\varphi e(\lambda) + 2\lambda b),\\
(\nabla_e S)(\xi,\xi) = 2[\varphi e(\lambda) -3\lambda e(\lambda) + 2\lambda b - 2\lambda^2 c]. \label{3.15}
\end{cases}
\eea
\bea
\begin{cases}
(\nabla_{\varphi e} S)(\xi,e) = 2\lambda^3 - f\lambda + c[e(\lambda) + 2\lambda c] -\varphi e(\varphi e(\lambda) + 2\lambda b),\\
(\nabla_\xi S)(\varphi e,e) = 0. \label{3.16}
\end{cases}
\eea
\bea
\begin{cases}
(\nabla_\xi S)(e,e) = -\xi(f),\\
(\nabla_e S)(\xi,e) = f - 2 - e(\varphi e(\lambda) + 2\lambda b) - b[e(\lambda) + 2\lambda c]. \label{3.17}
\end{cases}
\eea
\bea
\begin{cases}
(\nabla_\xi S)(e,\varphi e) = 0,\\
(\nabla_e S)(\xi,\varphi e) = 2\lambda^3 - f\lambda + b[\varphi e(\lambda) + 2\lambda b] - e(e(\lambda) + 2\lambda c). \label{3.18}
\end{cases}
\eea
We now complete the proof by substituting the equations \eqref{3.13}-\eqref{3.18}
in the equations \eqref{3.7}-\eqref{3.12} respectively.
\end{proof}

\begin{proposition} \label{prop3.2}
If the Reeb vector field of a non-Kenmotsu almost Kenmotsu 3-$h$-manifold $M^3$ is an eigen vector of the Ricci operator, then  $M^3$ is locally isometric to a non-unimodular Lie group equipped with a left invariant non-Kenmotsu almost Kenmotsu structure.
\end{proposition}
\begin{proof}
Since $\xi$ is an eigen vector of $Q$, then Lemma \ref{lem2.3} implies
\bea
\begin{cases}
\varphi e(\lambda) + 2\lambda b = 0,\\
e(\lambda) + 2\lambda c = 0. \label{3.19}
\end{cases}
\eea
It is well known that $$\frac{1}{2}X(r) = (\operatorname{div}Q)X = \sum_{i=1}^3 g((\nabla_{e_i} Q)X,e_i).$$
From the preceding equation, we can write
\bea
\frac{1}{2}X(r) = (\nabla_\xi S)(X,\xi) + (\nabla_e S)(X,e) + (\nabla_{\varphi e} S)(X,\varphi e). \label{3.20}
\eea
Making use of \eqref{2.6}, \eqref{3.19} and $\xi(\lambda) = 0$, we obtain the followings:
\be
(\nabla_\xi S)(\xi,\xi) = 0,\; (\nabla_e S)(\xi,e) = f - 2,\; (\nabla_{\varphi e} S)(\xi,\varphi e) = f - 2, \label{3.21}
\ee
\be
(\nabla_\xi S)(e,\xi) = 0,\; (\nabla_e S)(e,e) = 4\lambda b - e(f),\; (\nabla_{\varphi e} S)(e,\varphi e) = - 4\lambda b, \label{3.22}
\ee
\be
(\nabla_\xi S)(\varphi e,\xi) = 0,\; (\nabla_e S)(\varphi e,e) = - 4\lambda c,\; (\nabla_{\varphi e} S)(\varphi e,\varphi e) = 4\lambda c - \varphi e(f).\label{3.23}
\ee
Now, substituting $X = \xi,\;e\;\mathrm{and}\;\varphi e$ in \eqref{3.20} and then using \eqref{3.21}, \eqref{3.22} and \eqref{3.23} respectively, we obtain
\bea
\xi(r) = 4(f - 2),\;\; e(r) = -2e(f),\;\; \varphi e(r) = -2\varphi e(f). \label{3.24}
\eea
Using \eqref{3.19} and $\xi(\lambda) = 0$, we get from \eqref{2.5}
\bea
\xi(r) = -2\xi(f),\; e(r) = -2e(f) + 8\lambda^2 c,\; \varphi e(r) = -2\varphi e(f) + 8\lambda^2 b. \label{3.25}
\eea
Since $\lambda$ is a positive function, then the second and third equations of \eqref{3.24} and \eqref{3.25} implies $b = c = 0$. From Lemma \ref{lem2.3}, we get $f = 2$. Also from \eqref{3.19}, we get $e(\lambda) = \varphi e(\lambda) = 0$ and therefore $\lambda$ is a constant. Now, the Lie brackets given in \eqref{2.4} reduces to
$$[e,\xi] = e - \lambda \varphi e,\;\; [e,\varphi e] = 0\;\; \mathrm{and}\;\; [\varphi e,\xi] = - \lambda e + \varphi e.$$
Therefore, according to Milnor \cite{milnor}, $M^3$ is locally isometric to a non-unimodular Lie group equipped with a left invariant non-Kenmotsu almost Kenmotsu structure.
\end{proof}
Combining Lemma \ref{lem3.1} and Proposition  \ref{prop3.2}, the components of the Cotton tensor can be written as given in the follwing Corollary.
\begin{corollary} \label{cor3.3}
If the Reeb vector field of a non-Kenmotsu almost Kenmotsu 3-$h$-manifold $M^3$ is an eigen vector of the Ricci operator, then the components of the $(0,2)$-Cotton tensor $C$ with respect to an orthonormal frame $\{\xi,e,\varphi e\}$ on $M^3$ can be expressed as follows:
$$ C_{11} = C(\xi,\xi) = 0,\;  C_{12} = C(\xi,e) = 0,\;  C_{13} = C(\xi,\varphi e) = 0,$$
$$ C_{22} = C(e,e) = 2\lambda^3 - 2\lambda,\; C_{23} = C(e,\varphi e) = 0,\; C_{33} = C(\varphi e,\varphi e) = - 2\lambda^3 + 2\lambda.$$
\end{corollary}

We first consider the Cotton soliton with potential vector field $V$  pointwise collinear with the Reeb vector field. In this regard, we prove the following non-existing result.

\begin{theorem} \label{thm3.4}
On a non-Kenmotsu almost Kenmotsu 3-$h$-manifold such that the Reeb vector field is an eigen vector of the Ricci operator, there exist no Cotton soliton with potential vector field pointwise collinear with the Reeb vector field.
\end{theorem}
\begin{proof}
Suppose that the potential vector field $V$ is pointwise collinear with the Reeb vector field $\xi$. Then there exist a non-zero smooth function $\alpha$ on $M^3$ such that $V = \alpha \xi$. Now, substituting $X = e$ and $Y = \varphi e$ in \eqref{1.4} and using Lemma \ref{lem2.2} and Corollary \ref{cor3.3}, we get $2\lambda \alpha = 0$. This gives either $\lambda = 0$ or $\alpha = 0$. In either cases, we get a contradiction. This completes the proof.
\end{proof}

From Theorem \ref{thm3.4} and Proposition \ref{prop3.2}, we can say the following:

\begin{corollary} \label{cor3.5}
On a 3-dimensional non-unimodular Lie group equipped with a left invariant non-Kenmotsu almost Kenmotsu structure, there exist no Cotton soliton with potential vector field pointwise collinear with the Reeb vector field.
\end{corollary}

It is now quite tempting to consider the potential vector field $V$ as orthogonal to the Reeb vector field. In this setting, we prove the following:

\begin{theorem} \label{thm3.6}
Let $(M^3,g)$ be a non-Kenmotsu almost Kenmotsu 3-$h$-manifold such that the Reeb vector field is an eigen vector of the Ricci operator. If $g$ is a Cotton soliton with potential vector field orthogonal to the Reeb vector field, then $M^3$ is locally isometric to $\mathbb{H}^2(-4) \times \mathbb{R}$ and the Cotton soliton is steady. 
\end{theorem}
\begin{proof}
For a non-Kenmotsu almost Kenmotsu 3-$h$-manifold such that the Reeb vector field is an eigen vector of the Ricci operator, Proposition \ref{prop3.2} gives  $b = c = 0,\; f = 2,\; \lambda = \mathrm{constant}\; \mathrm{and}\; r = \mathrm{constant}$. Since $V$ is orthogonal to $\xi$, then there exist two smooth functions $\alpha_1$ and $\alpha_2$ on $M^3$ such that $V = \alpha_1 e + \alpha_2 \varphi e$. With the help of Lemma \ref{lem2.2}, we now obtain the components of $\mathcal{L}_V g$ as follows:
\bea
\begin{cases}
(\mathcal{L}_V g)(\xi,\xi) = 0,\; (\mathcal{L}_V g)(\xi,e) = \xi(\alpha_1) - \alpha_1 + \lambda \alpha_2,\\
(\mathcal{L}_V g)(\xi,\varphi e) = \xi(\alpha_2) - \alpha_2 + \lambda \alpha_1,\; (\mathcal{L}_V g)(e,e) = 2e(\alpha_1),\\
(\mathcal{L}_V g)(e,\varphi e) = e(\alpha_2) + \varphi e(\alpha_1),\; (\mathcal{L}_V g)(\varphi e,\varphi e) = 2\varphi e(\alpha_2). \label{3.26}
\end{cases}
\eea
We now use Corollary \ref{cor3.3} and the set of equations \eqref{3.26}. Substituting $X = Y = \xi$ in \eqref{1.4}, we get $\sigma = 0$. This shows that the Cotton soliton is steady. Now, substitution of $X = \xi$, $Y = e$ in \eqref{1.4} yields
\bea
\xi(\alpha_1) - \alpha_1 + \lambda \alpha_2 = 0. \label{3.27}
\eea
Replacing $X$ by $\xi$ and $Y$ by $\varphi e$ in \eqref{1.4}, we get
\bea
\xi(\alpha_2) - \alpha_2 + \lambda \alpha_1 = 0. \label{3.28}
\eea
Putting $X = Y = e$ in \eqref{1.4}, we obtain
\bea
2e(\alpha_1) + 2\lambda^3 - 2\lambda = 0. \label{3.29}
\eea
Substitution of $X = e$ and $Y = \varphi e$ in \eqref{1.4} yields
\bea
e(\alpha_2) + \varphi e(\alpha_1) = 0. \label{3.30}
\eea
Putting $X = Y = \varphi e$ in \eqref{1.4}, we infer
\bea
2\varphi e(\alpha_2) - 2\lambda^3 + 2\lambda = 0. \label{3.31}
\eea
Since $b = c = 0$, the Lie brackets given in \eqref{2.4} reduces to
\bea
[e,\xi] = e - \lambda \varphi e,\;\; [e,\varphi e] = 0\;\; \mathrm{and}\;\; [\varphi e,\xi] = - \lambda e + \varphi e. \label{3.32}
\eea
Since $\lambda$ is a positive constant, then from \eqref{3.27} and \eqref{3.29}, we obtain $$e(\xi(\alpha_1)) = e(\alpha_1) - \lambda e(\alpha_2)\;\; \mathrm{and}\;\; \xi(e(\alpha_1)) = 0.$$
Applying the first Lie bracket of \eqref{3.32} in the preceding eqaution, we get $\varphi e(\alpha_1) = e(\alpha_2)$. Hence, equation \eqref{3.30} implies $\varphi e(\alpha_1) = e(\alpha_2) = 0$. Now, from \eqref{3.28}, we get $e(\xi(\alpha_2)) = - \lambda e(\alpha_1)$. Also, we have $\xi(e(\alpha_2)) = 0$. Again, using these two in the first Lie bracket of \eqref{3.32} yields $\varphi e(\alpha_2) = e(\alpha_1)$. Applying \eqref{3.29} and \eqref{3.31} in the preceding relation and using the fact that $\lambda$ is a positive function, we obtain $\lambda = 1$. Now, it is easy to check that $\nabla Q = 0$. Notice that, a Riemannian 3-manifold is Ricci parallel if and only if it is locally symmetric. The rest of the proof follows from Theorem 5 of \cite{cho1} which says that ``A  non-Kenmotsu almost Kenmotsu 3-manifold $M^3$ is locally symmetric if and only if $M^3$ is locally isometric to the product space $\mathbb{H}^2(-4) \times \mathbb{R}$".
\end{proof}

As a combination of Proposition \ref{prop3.2} and Theorem \ref{thm3.6}, we have the following:
\begin{corollary}
If $g$ is a Cotton soliton with potential vector field orthogonal to the Reeb vector field on a 3-dimensional non-unimodular Lie group $M^3$ equipped with a left invariant non-Kenmotsu almost Kenmotsu structure, then $M^3$ is locally isometric to $\mathbb{H}^2(-4) \times \mathbb{R}$ and the Cotton soliton is steady. 
\end{corollary}

\subsection*{Acknowledgement} The author Dibakar Dey is thankful to the Council of Scientific and Industrial Research, India (File no: 09/028(1010)/2017-EMR-1) for their assistance.

\end{document}